\documentclass {amsart}
\pdfoutput = 1

\usepackage [utf8] {inputenc}
\usepackage {amsfonts}
\usepackage {amssymb}
\usepackage {amsmath}
\usepackage {graphicx}
\usepackage[font=small, textfont=it,labelfont=bf,
labelsep=endash]{caption}
\usepackage[foot]{amsaddr}
\usepackage{subfig}
\usepackage{color}
\usepackage{tikz}
\usetikzlibrary{shapes, patterns, calc}
\usepackage{./ricochet/ricochet} 

\newcommand{\IS}[2]{{\sc ice slide$_{#1}^{#2}$}}

\newcommand{\opt}[3]{\iota\sigma_{#1}^{#2}(#3)}
\newcommand{\opts}[2]{\iota\sigma_{#1}^{#2}}

\newtheorem{dfn}{Definition}
\newtheorem{prop}[dfn]{Proposition}
\newtheorem{theo}[dfn]{Theorem}
\newtheorem{cor}[dfn]{Corollary}

\newtheorem{rem}[dfn]{Remark}
\newtheorem{lem}[dfn]{Lemma}
\newtheorem{obs}[dfn]{Observation}

\author[P. Dorbec]{Paul Dorbec}
\email{paul.dorbec@labri.fr}
\author[E. Duchene]{\'Eric Duch\^ene}
\email{eric.duchene@univ-lyon1.fr}
\author[A. Fabbri]{Andr\'e Fabbri}
\email{andre.fabbri@etu.univ-lyon1.fr}
\author[J. Moncel]{Julien Moncel}
\email{julien.moncel@iut-rodez.fr}
\author[A. Parreau]{Aline Parreau}
\email{aline.parreau@univ-lyon1.fr}
\author[E. Sopena]{\'Eric Sopena}
\email{eric.sopena@labri.fr}

\address[P. Dorbec and E. Sopena]{
Univ. Bordeaux, LaBRI, UMR5800, F-33400 Talence, France.}
\address{CNRS, LaBRI, UMR5800, F-33400 Talence, France.}
\address[E. Duch\^ene, A. Fabbri, A. Parreau]{
Universit\'e de Lyon, CNRS.}
\address{Universit\'e Lyon 1, LIRIS, UMR 5205, F-69622, France.}
\address[J. Moncel]{CNRS, LAAS, Universit\'e de Toulouse, F-31077, France.}

\title {Ice sliding games}

\begin{document}

 \thanks{This research is supported by the ANR-14-CE25-0006 project of the French National Research Agency.}

\maketitle

\begin{abstract}
This paper deals with sliding games, which are a variant of the better known {\sc pushpush} game. On a given structure (grid, torus...), a robot can move in a specific set of directions, and stops when it hits a block or boundary of the structure. The objective is to place the minimum number of blocks such that the robot can visit all the possible positions of the structure. In particular, we give the exact value of this number when playing on a rectangular grid and a torus. Other variants of this game are also considered, by constraining the robot to stop on each case, or by replacing blocks by walls.
\end{abstract}

\maketitle
~

\section{Background and definitions}

Sliding/pushing puzzles are classical problems used for entertainment. In a sliding puzzle, entities (often described as robots) are moving around on a grid, and trying to reach a final position. Everytime a robot starts a move in a direction, it slides and cannot stop until it hits another element on the grid (a wall, a block or another robot). In a pushing puzzle, the entities often may stop a move without hitting a wall, but mostly they are also allowed to move some inert blocks by pushing them. Pushing games are known for example as Sokoban, but can also appear as enigmas in video games such as the Zelda sequel. They have been thoroughly studied before, and we refer the reader to the pleasant survey of Demaine ~\cite{Dema01} which also provides a classification of these games. He in particular gives a very detailed overview of known complexity results about such games. Without any surprise, a large majority of these are proved to be NP- or PSPACE-complete. A difference is also made between {\sc push} and {\sc pushpush}-games: in the first game the blocks are pushed one square at a time, while in the second they slide until they meet an obstacle whenever they are pushed.

Sliding games are still less studied, though there are many commercial games using this principle. {\sc Rasende roboter}, {\sc lunar lockout} (marketed in 1999 by Binary Arts) or its predecessor UFO are such examples. Before describing these games, we propose a general definition. An \IS{}{} game is a puzzle where one or more robots are on a grid, trying to reach a flag (the name was initially given in \cite{TCS12}). Each move of a robot consists in sliding in one direction until it meets an obstacle, which stops him. Obstacles include static obstacles, such as walls or blocks, but also other robots. In such a setting, the natural question is whether there is a sequence of moves for one robot to reach the flag (see Figure \ref{fig:intro} for an example of this game).

\begin{center}
\begin{figure}[h!]
\resizebox{4cm}{!}{
\input{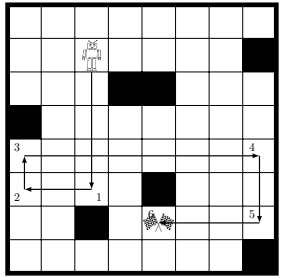}
}
\caption{Robot sliding game: a solution with six moves}
\label{fig:intro}
\end{figure}
\end{center}

This question, where obstacles are predefined, and a single robot is trying to reach a final position was posed recently by K. Burke, on his excellent blog devoted to combinatorial game theory \cite{Burk12}. K. Burke also asked for the minimum number of moves for the robot to reach the goal.
Actually, the resolution of this problem is not very hard, as explained in \cite{TCS12}. Indeed, it reduces to the search for a shortest path between the starting and final squares in an underlying subgraph of the grid. By way of consequence, a simple application of Djikstra's algorithm yields to a polynomial computation of the minimum number of moves of a solution (provided the game representation is not sublinear in the number of squares).

 Yet, another natural question arises when dealing with this problem: what if the player can choose the positions of the obstacles on the board? Of course, if the initial and final squares are fixed, at most one block is necessary, yielding to a solution with at most two moves. But if we ask the player to be able to move the robot to any position of the grid, the problem becomes considerably more challenging. More precisely, given an initial square $P$ of the grid, we ask what is the minimum number of blocks needed (and also how they must be placed) to move the robot from $P$ to any other square. (In that context, the question of the minimum number of moves becomes secondary.) \\

In correlation with this problem, we have identified two research parameters that slightly change the rules detailed above but raise new reflexion on the topic:
\begin{itemize}
\item Blocks can be substituted by walls in the problem. In other words, the robot bounces when hitting an edge of the grid (instead of a square).
\item One can suppose the final square not to be a flag to go through, and that the robot must stop on it. Note that it forces any square of the grid to be adjacent to a block (or a wall).
\end{itemize}
In the following, we differentiate four situations of the game, depending on the precise question and the nature of the obstacles. We denote the game \IS{B/W}{P/S}-K using a superscript P or S depending whether we require that the robot simply {\bf P}asses or need to {\bf S}top on the flagged position, a subscript B or W depending whether the obstacles are {\bf B}locks or {\bf W}alls, and K for the number of robots that we are allowed to move.

As a first example, the marketed game {\sc Rasende roboter}  is related to the game \IS{W}{S}-4 with a particular setting of the obstacles. In {\sc Rasende roboter}, the objective is to find the minimum number of moves of a solution where the initial and final positions are fixed. Its originality is that three additional robots can also be moved and be used as blocks on which bouncing could help the player.
{\sc Rasende roboter} had a good marketing success and also raised the curiosity of researchers in algorithmic game theory. Because of the presence of other robots and the size of the grid ($16\times 16$), simple shortest path algorithms are no more efficient. To the best of our knowledge, only AI algorithms based on multi-agent systems have been proposed to solve the game \cite{BuLR06}.

Another popular game called {\sc lunar lockout} (marketed in 1999 by Binary Arts) is very close to \IS{B}{S}-K, with the of a board with no bounding walls. As in {\sc Rasende roboter}, other robots can be used as movable blocks. Before its marketing, this puzzle was first introduced in 1998 by Yoshigahara and called {\sc UFO}. It was studied in the literature under this original name. Hock proved in \cite{Hock01} that deciding the existence of a solution in {\sc UFO} is NP-complete. In \cite{HaLi03}, a variant of {\sc UFO} is proposed, called {\sc GLLV}, where fixed robots are allowed and played on a general rectangular grid. This version clearly meets \IS{B}{S}-K. Whereas the PSPACE-completeness of {\sc UFO} is open, it is proved in \cite{HaLi03} that the variant which includes fixed blocks is PSPACE-complete.

\medskip

Note that the two games \IS{B/W}{P}-2 played on a rectangular grid are trivial, in the sense that no block is needed. Engels and Kamphans \cite{EK05} proved that the two games \IS{B/W}{S}-3 are also trivial in any rectangular grid. Hence the most interesting instances, at least on rectangular grids, seem to be \IS{}{}-1 and \IS{B/W}{S}-2. In this paper, we will focus in this paper on the games \IS{}{}-1, using a single robot, and denote it simply by \IS{}{}. We use the following definition.
\\

\begin{dfn}
\label{opt}
Let $G_{n,m}$ be a rectangular grid with $n$ rows and $m$ columns. Each square of the grid is denoted by a pair $(i,j)$ with $1\leq i \leq n$ and $1\leq j\leq m$. The parameter $\opt{B/W}{P/S}{G_{n,m}}$ is the minimum number of {\bf B}locks (resp. {\bf W}alls) that needs to be placed on the grid so that, from the starting position $(1,1)$ and for every position $(i,j)$ of $G_{n,m}$ which does not contain a block, there exists a sequence of moves making the robot {\bf P}ass over (resp. {\bf S}top on) $(i,j)$.
\end{dfn}

\begin{rem}
In what follows, we shall see that starting from the position $(1,1)$ is not very restrictive. Indeed, our constructions often guarantee that from any starting position, it is possible to move the robot and stop on $(1,1)$.
\end{rem}

\begin{rem}
In the definition, one can replace the rectangular grid by any kind of grids for which the moves of the robot is clear. We will for example consider in Section \ref{sec:BP} king grids and tori.
\end{rem}

\begin{figure}[h]
\begin{minipage}[c]{.49\linewidth}
\label{fig:intro_blocks}
\begin{center}
\resizebox{3cm}{!}{
\begin{board}{4}{4}
 \block{b1}{1}{3}
 \block{b2}{4}{2}
 \robot{r1}{1}{4}
 \draw[dash pattern=on 2mm off 1mm] (1cm - 0.5cm ,4cm - 0.5cm) -- (4cm - 0.5cm,4cm - 0.5cm) -- (4cm - 0.5cm,3cm - 0.5cm) -- (2cm - 0.5cm,3cm - 0.5cm) -- (2cm - 0.5cm,1cm - 0.5cm) -- (1cm - 0.5cm,1cm - 0.5cm) -- (1cm - 0.5cm,2cm - 0.5cm) -- (3cm - 0.5cm,2cm - 0.5cm) -- (3cm - 0.5cm,1cm - 0.5cm) -- (4cm - 0.5cm,1cm - 0.5cm);
\walltour
\end{board}
}
\caption{$\opt{B}{P}{G_{4,4}}\leq 2$.}
\end{center}
\end{minipage} \hfil
\begin{minipage}[c]{.49\linewidth}
\begin{center}
\resizebox{3cm}{!}{
\begin{board}{4}{4}
 \wall{w1}{0}{1}{1}{1}
 \wall{w2}{0}{3}{1}{3}
 \wall{w3}{1}{2}{2}{2}
 \wall{w4}{2}{0}{2}{2}
 \wall{w5}{2}{1}{3}{1}
 \wall{w6}{2}{3}{2}{4}
 \wall{w7}{3}{2}{4}{2}
 \wall{w8}{3}{2}{3}{3}
 \robot{r1}{1}{4}
 \draw[dash pattern=on 2mm off 1mm] (1cm - 0.5cm ,4cm - 0.5cm) -- (2cm - 0.5cm,4cm - 0.5cm) -- (2cm - 0.5cm,3cm - 0.5cm) -- (1cm - 0.5cm,3cm - 0.5cm) -- (1cm - 0.5cm,2cm - 0.5cm) -- (2cm - 0.5cm,2cm - 0.5cm) -- (2cm - 0.5cm,1cm - 0.5cm) -- (1cm - 0.5cm,1cm - 0.5cm);
 \draw[dash pattern=on 2mm off 1mm] (2cm - 0.5cm ,3cm - 0.5cm) -- (3cm - 0.5cm ,3cm - 0.5cm) -- (3cm - 0.5cm ,4cm - 0.5cm) -- (4cm - 0.5cm ,4cm - 0.5cm) -- (4cm - 0.5cm ,3cm - 0.5cm);
 \draw[dash pattern=on 2mm off 1mm] (3cm - 0.5cm ,3cm - 0.5cm) -- (3cm - 0.5cm ,2cm - 0.5cm) -- (4cm - 0.5cm ,2cm - 0.5cm) -- (4cm - 0.5cm ,1cm - 0.5cm) -- (3cm - 0.5cm ,1cm - 0.5cm);
\walltour
\end{board}
}
\caption{$\opt{W}{S}{G_{4,4}}\leq 9$.}
\label{fig:intro_mur}
\end{center}
\end{minipage}
\end{figure}


We proceed with the resolution of these problem for several instances, which should also stress how this game is related to some well-known problems in graph theory.
We first give in Section~\ref{sec:graph} the definitions and results of related graph problems. In Section~\ref{sec:BP}, we study $\opts{B}{P}$ on different kinds of grids, such as rectangular grids, king grids or tori. In Section~\ref{sec:variant}, the three other variants of the game are explored.

\section{Graph parameters related to \IS{}{}}\label{sec:graph}
The current section presents two optimization problems in graph theory which are in direct correlation with our game.  \\

\noindent
{\sc {Domination in graphs}}

\noindent
Given a graph $G$, a set $D\subseteq V(G)$ is said to be a dominating set (resp. a total dominating set) of $G$ if every vertex of $V(G)\setminus D$ (resp. $V(G)$) is adjacent to a vertex of $D$.

\begin{dfn}
Let $G$ be a graph. The value $\gamma(G)$ (resp. $\gamma_t(G)$), called {\it domination number} (resp. {\it total domination number}) of $G$, corresponds to the minimum cardinality of a dominating set (resp. total dominating set) in $G$.
\end{dfn}

In the next section, we observe that the dominating set problem is directly related to \IS{B}{S}. Indeed, the ability to stop the robot anywhere means that a block must be adjacent to any square of the grid. In other words, the set of blocks needs to be a dominating set of the grid. Total dominating sets appear in the study of the game \IS{B}{P}, as explained further in Subsection~\ref{sec:grids}.\\

Given any graph $G$, the computation of $\gamma(G)$ and $\gamma_t(G)$ are known to be NP-hard problems. However, their values are known for simple classes of graphs, such as grids or paths. As needed later, we give below the total domination number of paths.

\begin{prop}[Klobucar \cite{Klob04}]
Let $P_n$ denote the path with $n$ vertices. We have
\[
\gamma_t(P_n)=\left\{ \begin{array}{cl}
2\lfloor\frac{n}{4}\rfloor + 1 & \mbox{if } n\equiv 1\pmod 4\\
2\lfloor\frac{n}{4}\rfloor & \mbox{otherwise} \end{array} \right.
\]
\end{prop}

\bigskip
\noindent
{\sc {Edge cover in graphs}}

\noindent
Given a graph $G$, a set $S\subseteq E(G)$ is said to be an edge cover of $G$ if every vertex of $G$ is incident to at least one edge of $S$.

\begin{dfn}
Let $G$ be a graph. The value $\rho(G)$, called the {\it edge covering number} of $G$, corresponds to the size of a minimum edge cover of $G$.
\end{dfn}

Unlike the domination number, the value $\rho(G)$ can be computed in polynomial time for any graph $G$. In Subsections \ref{sec:iswp} and \ref{sec:isws}, we observe how edge covers interact with the game \IS{}{} played with walls. For example, we use the value $\rho(P_n)$, which is straightforwardly equal to $\lceil n/2\rceil$.

\section{\IS{B}{P} played on various grids}\label{sec:BP}
In all the paper, we consider grids with $n$ rows and $m$ columns. We denote by $R_i$ (resp. $C_j$) the $i^{th}$ row (resp. $j^{th}$ column) of $G$. Position $(i,j)$ is at the intersection of $R_i$ and $C_j$ and in the figures, unless something else is mentionned, $(1,1)$ is the top-left square.

\subsection{Rectangular grids}\label{sec:grids}
In this part, we consider the game \IS{B}{P} on a rectangular grid $G=G_{n,m}$: we place {\em blocks} on $G$ and the robot must be able to {\em pass} everywhere from the square $(1,1)$.

\begin{dfn}
Let $B$ be a set of blocks. The row $R_i$, $2\leq i\leq n-1$ (resp. column $C_j$, $2\leq j\leq m-1$) is said to be totally dominated by $B$ if there is at least one block of $B$ in row $R_{i-1}$ or $R_{i+1}$ (resp. in column $C_{i-1}$ or $C_{i+1}$).
\end{dfn}

If $B$ is a set of blocks in $G_{n,m}$ which totally dominates every column $C_j$, $2\leq j\leq m-2$, then, by moving blocks from columns $C_1$ and $C_m$ to columns $C_3$ and $C_{m-2}$ respectively, every column is still totally dominated by $B$. Hence we get the following:
\begin{obs}
\label{obs:gamma_tilde}
If $B$ is a set of blocks in $G_{n,m}$ which totally dominates every column $C_j$, $2\leq j\leq m-2$, then $|B|\geq \gamma_t(P_{m-2})$.
\end{obs}

\begin{prop}
\label{prop:grid}
For each $n\geq m$, we have $\opt{B}{P}{G_{n,m}}\geq \gamma_t(P_{m-2})$.
\end{prop}

\begin{proof}
Consider an optimal solution $B$ of \IS{B}{P} for $G_{n,m}$.
If every column $C_j$, $2\leq j\leq m-2$, is totally dominated by $B$ then, by Observation~\ref{obs:gamma_tilde}, we get $\opt{B}{P}{G_{n,m}}\geq \gamma_t(P_{m-2})$.

 Assume now that there exists a column $C_{j}$, $2\leq j\leq m-2$, which is not totally dominated by $B$. Let $I_j=\{i:(i,j)\in B\}$ be the line indices of the blocks of $C_j$.
 The robot can not initiate a vertical move in $C_j$ since to stop in $C_j$ it needs a block in column $C_{j-1}$ or $C_{j+1}$. Hence it has to go horizontally trough every row $R_i$, $i\in [2,n-1]\setminus I_j$. Therefore, each such row must be totally dominated by $B'=B\setminus \{(i,j):i\in I_j\}$. We define the set of blocks $B^*$ as follows:
 \[
 B^*=B' \cup \{(i-1,j):i\in I_j, i>1\}.
 \]
We claim that each row $R_i$, $2\leq i\leq n-1$, is totally dominated by the set $B^*$. Indeed, if $i\notin I_j$, then $R_i$ is totally dominated by $B'$ as observed above. Otherwise there is a block in row $R_{i-1}$. We finally get
\[
|B|\geq |B^*|\geq \gamma_t(P_{n-2}) \geq \gamma_t(P_{m-2})
\]
by Observation~\ref{obs:gamma_tilde} and the result follows.
\end{proof}

\begin{lem}
\label{lem:pattern_grid}
Assume that the pattern $P_n$ of size $n\times 8$ depicted on Figure \ref{fig:patterngrid} $(a)$ is in position $(1,j)$ ($(1,j)$ being the top-left square of the pattern). Then, if a robot is able to enter horizontally in each of the four positions $(1,j)$, $(n,j)$, $(2,j+7)$, and $(n-1,j+7)$, it can pass vertically through all the columns $C_j$ to $C_{j+7}$. In addition, if a robot enters horizontally the pattern $P_n$ in position $(1,j)$, it can leave it horizontally on rows $R_2$ and $R_{n-1}$ by going left and on rows $R_1$ and $R_n$ by going right (see Figure \ref{fig:patterngridleave}).
\end{lem}

\begin{proof}
The routes followed by the robot to satisfy the above lemma are shown on Figure \ref{fig:patterngrid}.
\end{proof}

\begin{figure}
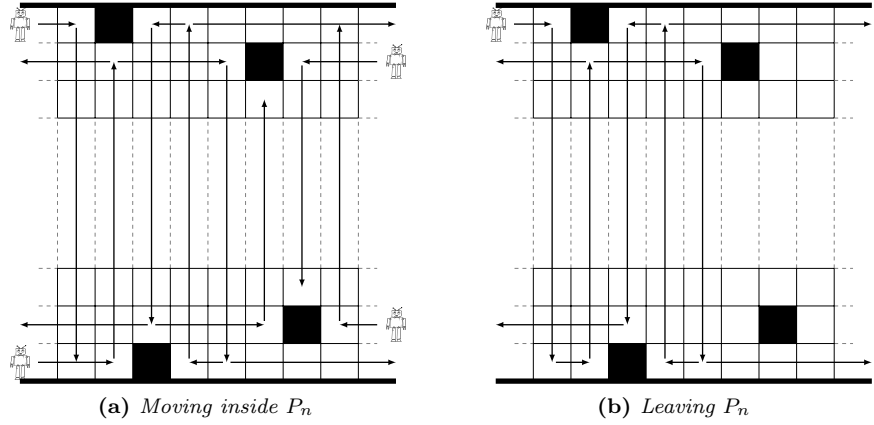

  \begin{center}
    \subfloat[Moving inside $P_n$]{
      \scalebox{0.5}{
        \input{carre-pattern}
      }
    }
    \hfil
    \subfloat[\label{fig:patterngridleave}Leaving $P_n$]{
      \scalebox{0.5}{
        \input{carre-pattern2}
      }
    }
    \caption{Pattern of size $n\times 8$ for \IS{B}{P} on grid and its representation.}
\label{fig:patterngrid}
\end{center}
\end{figure}

\begin{theo}\label{theo:grid}
For any $n\geq m$, $m\neq 10$, we have $\opt{P}{B}{G_{n,m}}=\gamma_t(P_{m-2})$.\\
For any $n\geq 10$, we have $\opt{P}{B}{G_{n,10}}=5$.
\end{theo}

\begin{proof}
Let $n\geq m > 10$. We consider the following cases:
\begin{itemize}
\item $m-2\equiv 3,4,5,6$ or $7\pmod 8$. The solution is build by gluing $\lfloor \frac{m-2}{8}\rfloor$ copies of $P_n$ starting from column $C_2$ (leaving column $C_1$ empty), the remaining columns being filled by the patterns of Figure \ref{fig:residus}. From Lemma \ref{lem:pattern_grid}, all the columns of all the copies of $P_n$ but the last one are visited. Figure \ref{fig:residus} shows that all the remaining rightmost columns can be visited by entering from the left on rows $R_1$ and $R_n$. The robot is also able to leave these columns from rows $R_2$ and $R_{n-2}$, so that the columns of the last copy of $P_n$ can be visited.
\item $m-2\equiv 1$ or $2\pmod 8$. Using $\lfloor \frac{m-10}{8}\rfloor$ copies of $P_n$, the proof is similar to the previous case using the patterns depicted in Figure \ref{fig:residus2}.
    \item $m-2\equiv 0\pmod 8$. The proof is again similar using $\lfloor \frac{m-18}{8}\rfloor$ copies of $P_n$ (recall that $m\neq 10$) and the pattern depicted in Figure \ref{fig:residus3}.
\end{itemize}

\begin{figure}[!]
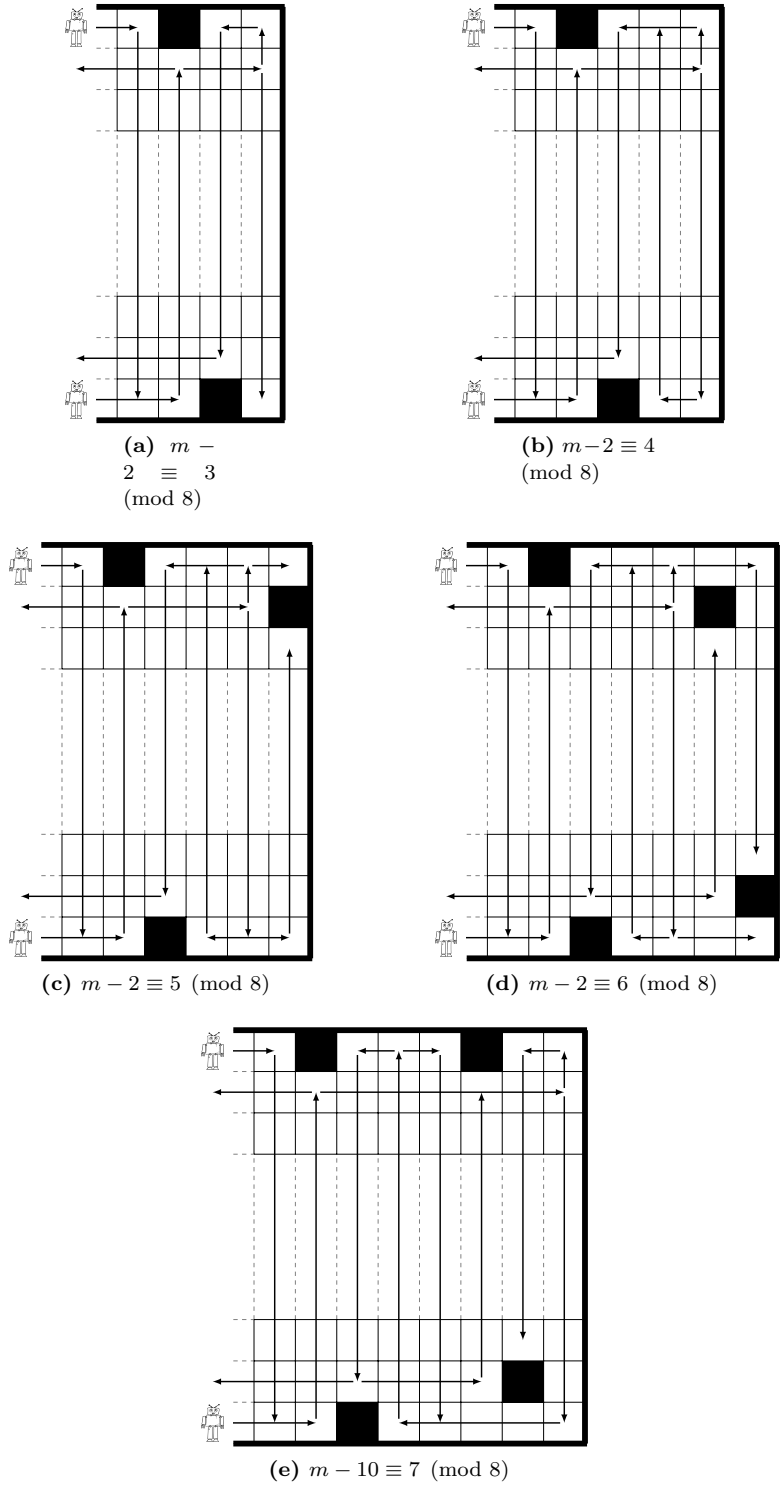

  \begin{center}
    \subfloat[$m-2\equiv 3\pmod 8$]{
      \scalebox{0.55}{
        \input{carre-mod3}
      }
    }
    \hfil
    \subfloat[$m-2\equiv 4\pmod 8$]{
      \scalebox{0.55}{
        \input{carre-mod4}
      }
    }
\hfil
    \subfloat[$m-2\equiv 5\pmod 8$]{
      \scalebox{0.55}{
        \input{carre-mod5}
      }
    }
    \hfil
    \subfloat[$m-2\equiv 6\pmod 8$]{
      \scalebox{0.55}{
        \input{carre-mod6}
      }
    }
\hfil
    \subfloat[$m-10\equiv 7\pmod 8$]{
      \scalebox{0.55}{
        \input{carre-mod7}
      }
    }
    \caption{Visiting the remaining columns when $m-2\equiv 3,4,5,6,7\pmod 8$}
\label{fig:residus}
\end{center}
\end{figure}

\begin{figure}
  \begin{center}
    \subfloat[$m-10\equiv 1\pmod 8$]{
      \scalebox{0.45}{
        \input{carre-mod1}
      }
    }
    \hfil
    \subfloat[$m-10\equiv 2\pmod 8$]{
      \scalebox{0.45}{
        \input{carre-mod2}
      }
    }
    \caption{Visiting the remaining columns when $m-10\equiv 1,2\pmod 8$}
\label{fig:residus2}
\end{center}
\end{figure}

We now consider the case $m<10$, $n\geq m$. If $m=1,2$, clearly no block is needed. For $m=3$, put a block in position $(n,3)$. For $m=4$, put two blocks in positions $(n,3)$ and $(1,4)$. For $5\leq m\leq 9$, leave column $C_1$ empty glued to the pattern of width $(m-1)$ given in Figure~\ref{fig:residus}.

In each case, the number of blocks used in our solutions is exactly $\gamma_t(P_{m-2})$.

\begin{figure}
\scalebox{0.45}{
\input{carre-mod0}
}
\caption{Visiting the remaining columns when $m-18\equiv 0\pmod 8$}
\label{fig:residus3}
\end{figure}

In the case $m=10$, one can easily show that five blocks are enough by using the pattern of Figure~\ref{fig:residus2} $(a)$ without the first column. It now remains to prove that $\gamma_t(P_8)=4$ blocks are not sufficient. Suppose on the contrary that $\opt{B}{P}{G_{n,10}}=\gamma_t(P_8)=4$ for $n\geq 10$, and consider an optimal solution $B$. According to the proof of Proposition~\ref{prop:grid}, one can assume that every column $C_j$, $2\leq j\leq m-2$ is totally dominated by $B$. There exists a unique minimum total dominating set in $P_8$, depicted in Figure~\ref{fig:tds_P8}. 
Consequently, the four blocks of $B$ are located in columns $C_3,C_4,C_7$ and $C_8$.

\begin{figure}[!h]
\scalebox{1}{
\begin{tikzpicture}[auto, thick]

\foreach \i in {1,...,8}
\node[sommet] (\i) at (\i,0) {};

\foreach \j in {2,3,6,7}
\node[sommet] () at (\j,0) {X};

\draw[arete] (1) -- (2) --(3)--(4)--(5)--(6)--(7)--(8);

\end{tikzpicture}
}
\caption{Minimum total dominating set in $P_8$}
\label{fig:tds_P8}
\end{figure}

Now, since column $C_2$ is totally dominated by column $C_3$, the block in $C_3$ is necessarily located on the first or last row. Otherwise, it would not be possible to make the robot stop in column $C_2$ since there is no block in $C_1$. By symmetry, the block in $C_8$ is also located on the first or last row. Now, with the same argument, since $C_3$ is totally dominated by $C_4$ and there are no blocks in $C_1$ and $C_2$, the only way to make the robot stop on column $C_3$ is to place the block of $C_4$ in the first or last row, and not adjacent to the block in $C_3$. Ditto for the block in $C_7$. There are thus only $4$ possible sets for $B$, but for each of them, the position $(1,5)$ is not reachable.
\end{proof}

\begin{rem}
  Note that our constructions preserves the accessibility of the robot to every square of the grid if the starting position is different from $(1,1)$. Indeed, one can verify that from any position in $P_n$ or in the ending patterns, there exists a sequence of moves landing in $(1,1)$.
\end{rem}

\subsection{Tori}
We now consider the torus grid $T_{n,m}$ and compute the value $\opt{B}{P}{T_{n,m}}$ for any values of $n$ and $m$. Indices are taken modulo $m$ for columns and modulo $n$ for rows.

We first give a lower bound, similar to Proposition~\ref{prop:grid}.
\begin{prop}
For any $n\geq m$, $\opt{B}{P}{T_{n,m}}\geq \frac{m-1}{2}$.
\end{prop}

\begin{proof}
Consider an optimal solution $B$ of \IS{B}{P} for $T_{n,m}$.
If for every column $C_j$ which is not the starting column of the robot, there is a block in column $C_{j-1}$ or $C_{j+1}$, then there are at least $\frac{m-1}{2}$ blocks in the solution.

Hence we assume that there exists a column $C_j$, not the starting column of the robot, such that there is no block in columns $C_{j-1}$ and $C_{j+1}$. Let $n_j$ be the number of blocks in column $C_j$. The robot can not go vertically in column $C_j$ since to stop in column $C_j$ it needs a block in column $C_{j-1}$ or $C_{j+1}$. Hence it has to go horizontally trough the $n-n_j$ free squares of the column. But to go horizontally on row $R_i$, either the robot starts on this row, or there is a block in row $R_{i-1}$ or $R_{i+1}$. Thus, to go through the $n-n_j$ free squares of $C_j$, the robot needs at least $\frac{n-n_j-1}{2}$ blocks. Note that these blocks are not on column $C_j$. At the end, the solution has at least $\frac{n-n_j-1}{2}+n_j\geq \frac{m-1}{2}$ blocks.
\end{proof}

We will prove that this lower bound is reached for $m\geq 6$. For that, we will prove by induction the following stronger statement:

\begin{prop}\label{prop:indtore}
For any $m\geq 6$, there exists a solution of \IS{P}{B} on $T_{m,m}$ with $\frac{m-1}{2}$ blocks, such that:
\begin{enumerate}
\item if $i_m$ is the maximum index of a row with a block, the robot can pass horizontally on the row $R_{i_m+1}$,
\item the robot can pass vertically on all the columns.
\end{enumerate}
\end{prop}

\begin{proof}
We will use for the induction the pattern of size $4\times 8$ of Figure~\ref{fig:patterntore}. The figure gives a proof of the following lemma:

\begin{lem}
Assume that the pattern of Figure \ref{fig:patterntore} is in position $(i,j)$ ($(i,j)$ is the top-left square of the pattern) and that there is no block in any column and row intersecting the pattern (i.e. rows $R_i$ to $R_{i+3}$ and columns $C_j$ to $C_{j+7}$). Then, if a robot enters in position $(i,j)$ horizontally in the pattern of Figure~\ref{fig:patterntore}, it can pass vertically on all the columns $C_j$ to $C_{j+7}$ and can go out of the pattern horizontally on rows  $R_{i-1}$ and $R_{i+4}$, in any direction.
\end{lem}

\begin{center}
\begin{figure}[h]
\scalebox{0.4}{
\input{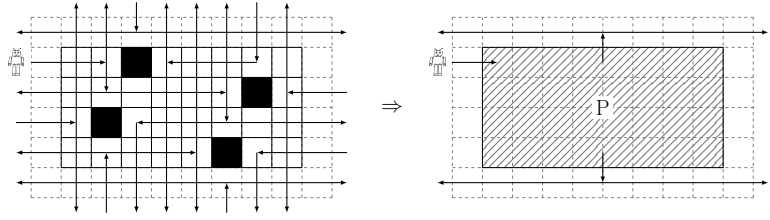}
}
\caption{Pattern for \IS{B}{P} on torus and its representation.}
\label{fig:patterntore}
\end{figure}
\end{center}

Using this pattern, if we know a solution for $T_{m,m}$ of size $\frac{m-1}{2}$ satisfying Conditions (1) and (2) of the proposition, we can get a solution for $T_{m+8,m+8}$ of size $\frac{m-1}{2}+4$ still satifsfying the conditions. Indeed, one can copy the solution of $T_{m,m}$ and if $i_m$ (resp. $j_m$) denotes the largest index of a row (resp. a column) containing a block, add the pattern in position $(i_m+1,j_m+1)$ (see Figure \ref{fig:inductiontore}).

\begin{center}
\begin{figure}
\scalebox{0.4}{
\input{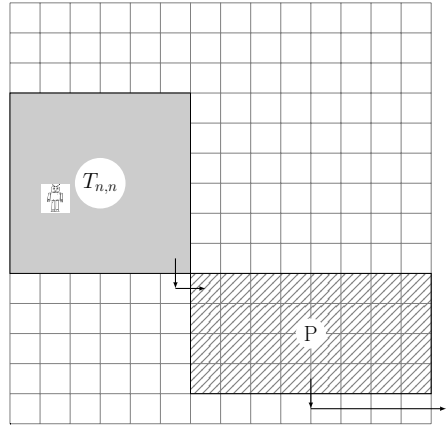}
}
\caption{Construction of a solution for $T_{m+8,m+8}$ from a solution of $T_{m,m}$. The conditions 1 and 2 of Proposition
\ref{prop:indtore} remain satisfied.}
\label{fig:inductiontore}
\end{figure}
\end{center}

Hence we just have to prove the proposition for $m=6$ to $m=13$. This is proved by Figure \ref{fig:basetore}. On this figure, we show the solutions for odd $m$. We get the solutions for even $m$ by removing the row and the column in gray.

\begin{figure}
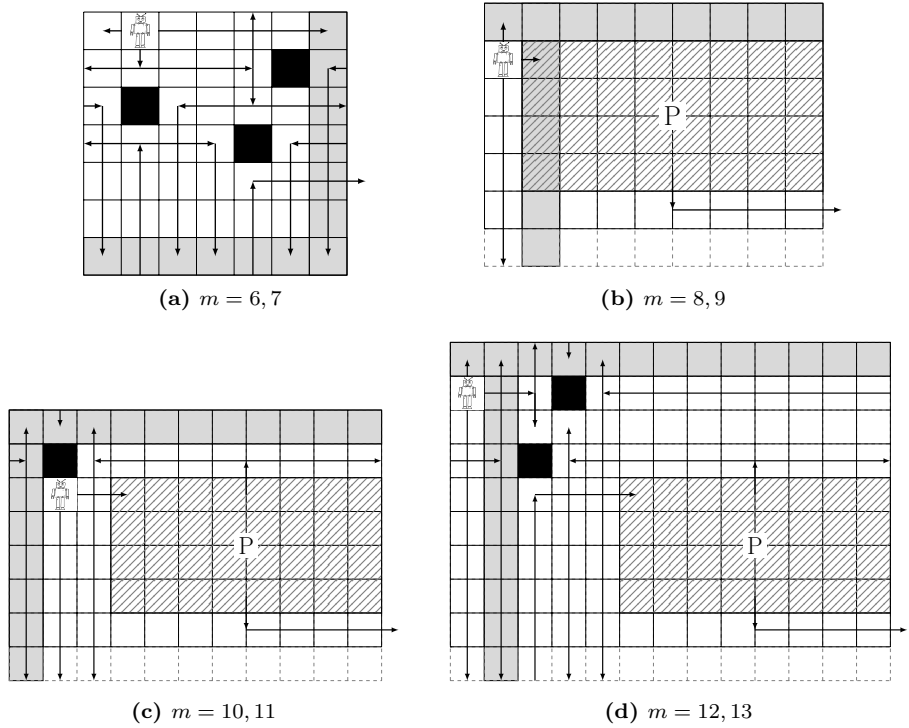

  \begin{center}
    \subfloat[$m=6,7$]{
      \scalebox{0.5}{
        \input{dessin_tore67}
      }
    }
    \hfil
    \subfloat[$m=8,9$]{
      \scalebox{0.5}{
        \input{dessin_tore89}
      }
    }
    \hfil
    \subfloat[$m=10,11$]{
      \scalebox{0.45}{
        \input{dessin_tore1011}
      }
    }
    \hfil
    \subfloat[$m=12,13$]{
      \scalebox{0.45}{
        \input{dessin_tore1213}
      }
    }
  \end{center}
  \caption{Small cases for induction for \IS{P}{B} on torus. Since all columns are passed vertically, we did not represent all the horizontal rows but just the significant ones. To get the even grids, remove the gray row and column.}
  \label{fig:basetore}
\end{figure}

\end{proof}

Since all the columns are passed vertically, we can add as many rows as we want to the torus and we get optimal solutions for $T_{n,m}$ with $n\geq m \geq 6$:

\begin{cor}
For $n\geq m\geq 6$, $\opt{P}{B}{T_{n,m}}=\frac{m-1}{2}$
\end{cor}

We now complete the study with the small values of $m$:
\begin{prop}
We have:
\begin{itemize}
\item $\opt{P}{B}{T_{n,1}}=0$
\item $\opt{P}{B}{T_{n,2}}=n-2$
\item $\opt{P}{B}{T_{n,3}}=3$
\item $\opt{P}{B}{T_{n,4}}=3$
\item $\opt{P}{B}{T_{n,5}}=3$
\end{itemize}
\end{prop}

\begin{proof}
The case $m=1$ is trivial.

If $m=2$, assume that the robot starts at position $(1,1)$. Then it can only go vertically on column 1, horizontally on row 1 and eventually horizontally on two more rows $i-1$ and $i+1$ if there is a block in position $(i,1)$.
Hence $n-3$ squares are not reachable and we need to add blocks on these squares.

For $m\in\{3,4,5\}$, assume there is a solution with starting position $(1,1)$ and two blocks in positions $(i_1,j_1)$ and $(i_2,j_2)$. We must have $1\in \{i_1,j_1,i_2,j_2\}$. Assume that $i_1=1$ (other cases are similar). We have $j_1\neq 1$, $j_1\neq j_2$ and $i_2\neq 1$ (otherwise the robot can not go everywhere). Then one can check that the robot can not go in position $(i_2,i_1)$, which leads to a contradiction.
Blocks in positions $\{(1,3),(2,4),(3,2)\}$ is a solution with three blocks.
\end{proof}

\begin{rem}
In the torus, all the positions are equivalent. Also, if we knows the starting position before placing the blocks, we can always use our constructions. However, if the solution must work for any starting position, then one need at least $\min(n,m)$ blocks. Indeed, if there exists a row $R_i$ and a column $C_j$ without any block, a robot starting in $(i,j)$ will never be able to leave the squares of $R_i\cup C_j$.
\end{rem}

\subsection{King grids}
We now consider the game \IS{B}{P} played on the King grid $n\times m$, denoted $\mathcal{K}_{n,m}$, also known as the strong product of two paths $P_n\boxtimes P_m$. In the King grid, the robot is also allowed to initiate moves diagonally.

\begin{theo}
Let $m\ge n$ be positive integers.

 If $\gcd(n,m)\le 3$, then $\opt{B}{P}{\mathcal{K}_{n+1,m+1}}=0$.
\end{theo}

\begin{proof}
Let $m\ge n$ be positive integers such that $\gcd(n,m)\le 3$, and consider the game on the King grid ${\mathcal{K}_{n+1,m+1}}$. We first observe that the robot can stop on all positions $(1,1+ \alpha n - \beta m)$ where  $0\le \alpha n - \beta m \le m$ (see Figure~\ref{fig:kinga}). From $(1,1)$, the robot can go diagonally to $(1+n,1+n)$ then up to $(1,n+1)$. Iterating the process, it reaches $(1, 1+ \alpha n)$ for all $0\le \alpha \le \frac{m}{n}$. Eventually, the robot reaches the position $(1,1+\lfloor\frac{m}{n}\rfloor n)$. Then, going diagonally, it reaches $(1+m-\lfloor\frac{m}{n}\rfloor n, m+1)$, then can slide to $(1+m-\lfloor\frac{m}{n}\rfloor n, 1)$, and can continue diagonally its movement to $(1+n,1+(\lfloor\frac{m}{n}\rfloor+1) n -m)$ from which it can go up to  $(1,1+(\lfloor\frac{m}{n}\rfloor+1) n -m)$. Iterating the same process, we get that it reaches all positions $(1,1+ \alpha n - \beta m)$ where  $0\le \alpha n - \beta m \le m$. Then, thanks to Bézout's identity, we deduce that the robot reaches all positions $(1,1+\alpha \gcd(m,n))$ where $0\le \alpha \le \frac{m}{\gcd(m,n)}$.

Now, if $\gcd(m,n)=1$, from these positions, the robot can pass vertically any position. If $\gcd(m,n)=2$, then the robot passes vertically all positions $(x, 1+2y)$. From $(1, 1+2y)$, it can also reach by a diagonal $(1+2y,1)$, and thus can pass horizontally all $(1+2x, y)$. Positions $(2x,2y)$ can be reached from either $(1+2(x-y), 1)$ (if $x\ge y$) or from $(1,1+2(y-x)$) (otherwise). So the robot can pass all positions. 
Finally, if $\gcd(m,n)=3$, the robot can pass vertically all $(x,1+3y)$ and horizontally all $(1+3x,y)$. The positions $(2+3x, 2+3y)$ and $(3x, 3y)$ can be reached diagonally from $(1+3(x-y),1)$ or from $(1,1+3(y-x))$. The positions $(3x,2+3y)$ and $(3x-1,3y)$ can be reached diagonally from $(n+1,3(y+x)-n+1)$ or from $(3(y+x)+1,1)$.
\end{proof}

\begin{figure}
  \begin{center}
    \subfloat[On $\mathcal{K}_{n+1,m+1}$ when $\gcd(n,m)=3$, the robot can stop on grey squares.]{
      \scalebox{0.55}{
        \input{king-ex1}
      }\label{fig:kinga}
    }
    \hfil
    \subfloat[How we use blocks]{
      \scalebox{0.45}{
        \input{king-block}
      } \label{fig:kingb}
    }
    \caption{Block pass in the king grid.}
\label{fig:king}
\end{center}
\end{figure}

\begin{theo}
For $m\ge n$ positive integers, we have $$\opt{B}{P}{\mathcal{K}_{n+1,m+1}}\le \left\lceil\frac{\gcd(n,m)-3}{8}\right\rceil\,.$$
\end{theo}

\begin{proof}
Let $k=\left\lceil\frac{\gcd(n,m)-3}{8}\right\rceil$, we prove now that using $k$ blocks, the robot may pass on every unoccupied position. For $0\le i<k$, we place a block $B_i$ on position $(n+1, n+1-2i)$ when $i$ is even, and $(1,n+1-2i)$ when $i$ is odd (see Figure~\ref{fig:king-20-20}). 

We first prove by induction that the robot can reach position $(1,n+1-j)$ or position $(n+1,n+1-j)$ for all $0\le j\le 2k$, and thus, with two diagonal moves and possibly one vertical move, the robot also reaches positions $(1,1+j)$ and $(n+1,1+j)$. Let $j$ be a positive even integer in the range $0\le j\le 2k-2$, and assume the robot can reach the positions $(1,1+j)$ and $(n+1, 1+j)$. The base case, when $j=0$, is obvious. If $j$ is even, from $(1,j+1)$, the robot can slide to $(j+1,1)$ then to $(n,n-j)$ where it get stopped by block $B_{\frac{j}{2}}$ (see Figure~\ref{fig:kingb}). Then it can go up to $(1,n-j)$ and diagonally to $(n+1,n-j-1)$, proving the property for $j+1$ and $j+2$. 

From these positions, with similar arguments as in the previous lemma, we get that the robot can always reach positions of type $(1,1+\alpha\gcd(n,m)-j)$ or $(n+1,1+\alpha\gcd(n,m)-j)$ as well as  $(1+\alpha\gcd(n,m)+j,1)$ or $(1+\alpha\gcd(n,m)+j,n+1)$ with $0\le j\le 2k$ (see Figure~\ref{fig:king-20-20}. Then, the robots can reach all positions whose line or column can be written in such a way. So consider now a position of type $(1+\alpha\gcd(n,m)+x,1+\beta\gcd(n,m)+y)$ where $x$ and $y$ are larger than $2k$ but less than $\gcd(n,m)-2k\le  6k+3$. If $|x-y| \le 2k$, then the robot can reach that position by sliding diagonally from $(1+(\alpha-\beta)\gcd(n,m)+(x-y),1)$ or $(1,1+(\alpha-\beta)\gcd(n,m)+(x-y))$. If $|x-y|\ge 2k+1$, we deduce that 
\begin{align*}
x+y-gcd(n,m) & \le 2(gcd(n,m)-2k-1)-(2k+1)-gcd(n,m)\le 2k \\
x+y-gcd(n,m) & \ge 2(2k+1)+2k+1 -gcd(n,m)\ge -2k
\end{align*}
Therefore, if  $y'=1+\alpha\gcd(n,m)+x + n+1-(1+\beta\gcd(n,m)+y)$ is non negative, then it can be writen in the form $1+\gamma\gcd(n,m)+j$ or $1+\gamma\gcd(n,m)-j$ where $0\le j\le 2k$, and the position can be reached diagonally from $(n+1,x')$, otherwise $x' = n+1 - y'$ and then the position can be reached from $(x',1)$.

\end{proof}

\begin{figure}
  \begin{center}
  \scalebox{0.5}{
       \input{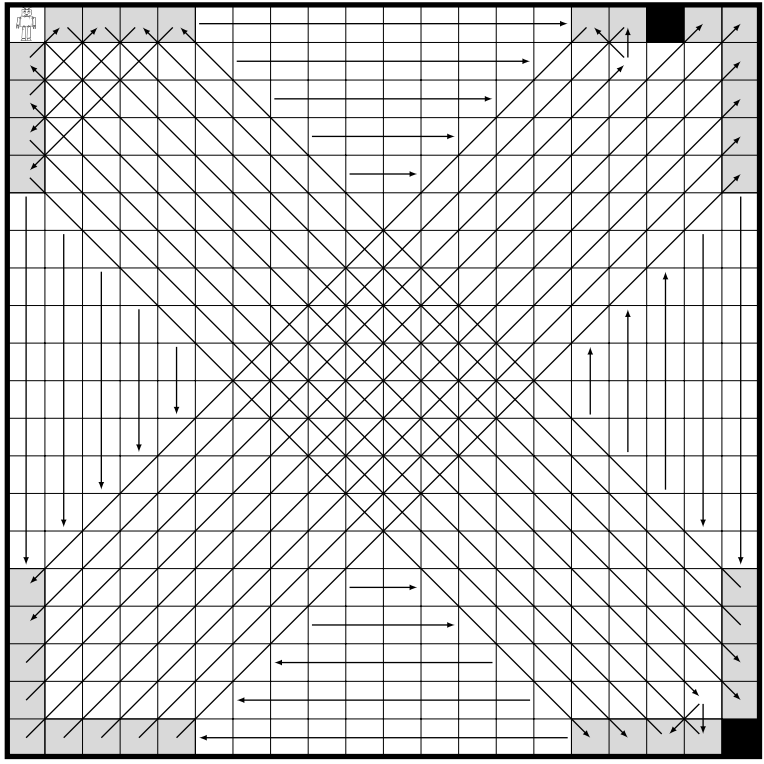}}
  \caption{Two blocks for the King grid $\mathcal{K}_{20,20}$.}
\label{fig:king-20-20}
\end{center}
\end{figure}

When $\gcd(m,n)>3$, it is not difficult to verify that there is no solution without any block, by checking all the possible moves. We can observe that there is no possibility to reach the position $(2,3)$ for example. However, we do not have any general proof that our bound is optimal, and we leave it as an open question.

\section{Other instances of \IS{}{}}\label{sec:variant}

\subsection{\IS{B}{S}}\label{sec:isbs}
In what follows, we consider the game \IS{}{} played with blocks on a grid where the robot has to stop on each square. An inner square of the grid is a square $(i,j)$ such that $2\leq i \leq n-1$ and $2\leq j\leq m-1$. Other squares are called border squares.

We obtain the following lower bound for $\opt{B}{S}{G_{n,m}}$:

\begin{prop}
For each $n\geq m \geq 3$, we have $\opt{B}{S}{G_{n,m}}\geq \gamma(G_{n-2,m-2}).$
\end{prop} 

\begin{proof}
Consider an optimal solution $B$ of \IS{B}{S} on $G_{n,m}$.
Let $(i,j)$ be an inner square of the grid.
There must a block adjacent to $(i,j)$ otherwise the robot cannot stop on it. 
Let $G'$ be the grid induced by the inner squares. It is isomorphic to $G_{n-2,m-2}$.
Let $B_1$ be the set of blocks of $B$ that are located on inner squares. 
Let $V_b$ be the set of squares of $G'$ that are not adjacent to a block of $B_1$. Note that all these squares are on the border of $G'$, and that, for each square $(i,j)$ of $V_b$, there is a block on an adjacent square of $(i,j)$ which is on the border of $G_{n,m}$. Thus we have $|B| \geq |B_1|+|V_b|$. Now let us consider $B'$ defined as the disjoint union of $B_1$ and $V_b$. Clearly, $B'$ is a dominating set of $G_{n-2,m-2}$, and we have $|B'|\leq |B|$. Hence we have:
$$\opt{B}{S}{G_{n,m}}=|B|\geq |B'|\geq \gamma(G_{n-2,m-2}).$$
\end{proof}

We now give a construction of a solution of \IS{B}{S} on $G_{n,m}$.

\begin{prop}\label{prop:cons-BS}
For any $n\geq m \geq 3$, we have  $$\opt{B}{S}{G_{n,m}}\leq \gamma(G_{n-2,m-2}) +\frac{18}{25}(m+n)+\frac{28}{5}.$$ 
\end{prop}

\begin{proof}
Let $n\geq m\geq 2$.

We construct a solution $B$ of \IS{B}{S} in three steps. The construction is illustrated on Figures \ref{fig:consBS} and \ref{fig:rectBS}.

Let first $B_1$ be the set of blocks that are the translation of the block located on square $(1,2)$ by linear combinations of the two vectors $\{(1,2), (3,1)\}$ (see Figure~\ref{fig:consBS}). Note that $B_1$ covers all the inner squares of $G_{n,m}$ and that each inner square is covered exactly once.

Let us call $B_1'$ the subset of blocks which are located on the border. By construction, there are at most $2\left(\left\lceil \frac{n-2}{5} \right \rceil+\left\lceil \frac{m-2}{5} \right\rceil\right) \leq \frac{2}{5}(m+n)+2$ such blocks. Clearly, the blocks of $B_1'$ cover a subset $S'$ of the inner squares such that $|B_1'|=|S'|$.

Let us denote $S''$ the set of inner squares which are not in $S'$. Since each inner square is covered exactly once, then the subset $B_1''$ of blocks of $B_1$ that are located on inner squares satisfies $5|B_1''|=|S''|$.

On the other hand, since each square of the grid has at most four neighbours, then we have $\gamma(G_{n-2,m-2}) \geq \frac{1}{5}(|S'|+|S''|)$. Hence we have $\gamma(G_{n-2,m-2}) \geq \frac{1}{5}|B_1'| + |B_1''|$, and since $B_1$ is the disjoint union of $B_1'$ and $B_1''$ then we have $|B_1| \leq \gamma(G_{n-2,m-2}) +\frac{8}{25}(m+n)+\frac{8}{5}$.

The blocks of $B_1$ will be enough to ensure that the robot will be able to stop on almost each inner square as soon as it may reach some of them. Indeed, if the robot enters in a rectangle of size $4\times 4$ that has four blocks on its border (see as an example the gray rectangle on Figure \ref{fig:consBS}), then it will be able to stop in the four inner squares of this rectangle and thus to go out of the rectangle from these four places (see Figure \ref{fig:rectBS}). To ensure that the robot will be able to enter in those rectangles from the border of the grid and to stop on the border, we add some blocks that will form the set $B_2$. 

\begin{center}
\begin{figure}[h]
\resizebox{5cm}{!}{
\input{dessin-BS-preuve1}
}
\caption{Solution of \IS{B}{S} on $G_{8,10}$}
\label{fig:consBS}
\end{figure}
\end{center}

\begin{center}
\begin{figure}[h]
\resizebox{3cm}{!}{
\input{BS-rectangle}
}
\caption{Pattern of the solution of \IS{B}{S}}
\label{fig:rectBS}
\end{figure}
\end{center}

For each block of $B_1$ located on an inner square which is neighbour of a border square, we add a block in $B_2$ in the following way:

\begin{itemize}
\item for each block $(2,j)\in B_1$, we add the block $(1,j)$ to $B_2$;
\item for each block $(n-1,j)\in B_1$, we add the block $(n,j-1)$ to $B_2$;
\item for each block $(i,2)\in B_1$, we add the block $(i,1)$ to $B_2$;
\item for each block $(i,m-1)\in B_1$, we add the block $(i,m)$ to $B_2$.
\end{itemize}

By construction, we have $|B_2|\leq 2\left(\left\lceil \frac{n-2}{5} \right \rceil+\left\lceil \frac{m-2}{5} \right\rceil\right) \leq \frac{2}{5}(m+n)+2$.

Finally, we add blocks in the corners of $G_{n,m}$ if these squares are not reachable by the robot:

\begin{itemize}
\item if there is a block on $(2,m-1)$, then we add the block $(1,m)$;
\item if there is a block on $(n-1,2)$, then we add the block $(n,1)$;
\item if there is a block on $(n-1,m-1)$, then we add the block $(n,m)$.
\end{itemize}

These blocks form the set $B_3$. Note that at most two blocks are added in $B_3$, since $(1,m)$ is added if and only if $m \equiv 0\bmod 5$, and $(n,1)$ is added if and only if $n \equiv 2\bmod 5$.

The set of blocks $B=B_1 \cup B_2 \cup B_3$ form a solution of \IS{B}{S}, of cardinality at most $\gamma(G_{n-2,m-2}) +\frac{18}{25}(m+n)+\frac{28}{5}$.

\end{proof}

By combining the lower and uppers bounds and since $\gamma(G_{n,m})=\left\lfloor \frac{(n+2)(m+2)}{5}\right \rfloor -4$ \cite{DomGrid}, we obtain an asymptotic value for $\opt{B}{S}{G_{n,m}}$:

\begin{cor}
If $m$ and $n$ are large enough, $\opt{B}{S}{G_{n,m}}=\frac{mn}{5}+O(m+n)$.
\end{cor}

Notice that the construction of Proposition~\ref{prop:cons-BS} is optimal for $n=m=2$, $n=m=3$, and $m=n=4$. We did not try to find the exact value of $\opt{B}{S}{G_{n,m}}$ in the general case since we think that such a result does not deserve much interest compared to the tedious case study it would need to be proved. 

\begin{rem}
Remark that, in the torus, using an optimal dominating set we immediately get $\opt{B}{S}{T_{n,m}}=\frac{nm}{5}$ when $n \equiv 0\bmod 5$ and $m \equiv 0\bmod 5$. Similarly, in the general case, one could also show that $\opt{B}{S}{T_{n,m}} = \frac{nm}{5} + O(m+n)$.
\end{rem}

\subsection{\IS{W}{P}}\label{sec:iswp}

In what follows, we consider the game \IS{}{} played with walls on a grid and where the robot has to pass over each square. Unlike blocks, there will be two types of walls: horizontal and vertical ones (see Figure \ref{fig:intro_mur}).

\begin{dfn}
Let $W$ be a set of walls. The row $R_i$, $2\leq i\leq n-1$ (resp. column $C_j$, $2\leq j\leq m-1$) is said to be covered by $W$ if there is at least one horizontal wall of $W$ adjacent to row $R_{i}$ (resp. one vertical wall adjacent to column $C_{i}$).
\end{dfn}

By similar considerations to Subsection~\ref{sec:grids}, we get the following observation:
\begin{obs}
\label{obs:rho}
If $W$ is a set of walls in $G_{n,m}$ which covers every column $C_j$, $2\leq j\leq m-1$, then $|W|\geq \rho(P_{m-2})$.
\end{obs}

This result leads to the following lower bound for $\opt{W}{P}{G_{n,m}}$:
\begin{prop}
For each $n\geq m$, we have $\opt{W}{P}{G_{n,m}}\geq \rho(P_{m-2})$.
\end{prop}

\begin{proof} Similar to the one of Proposition~\ref{prop:grid}. If every column is covered by a wall, then the proposition holds from Observation~\ref{obs:rho}. Now consider an optimal solution $W$ such that there exists a column $C_j$, $2\leq j\leq m-1$ which is not covered by a wall. The robot cannot move vertically in $C_j$ since there is no vertical wall adjacent to this column. Hence it has to go horizontally through every row $R_i$, $2\leq i\leq n-1$, implying that each such row must be covered by a wall. Therefore we get the desired result:
\[
|W|\geq \rho(P_{n-2}) \geq \rho(P_{m-2}).
\]
\end{proof}

The following lemma is the key to reach the lower bound of the above proposition:
\begin{lem}\label{lem:WP}
Assume that the pattern $P_n$ of size $n\times 4$ depicted on Figure \ref{fig:patternWP} is in position $(1,j)$ ($(1,j)$ being the top-left square of the pattern). Then, if a robot is able to enter horizontally either in $(1,j)$ or $(n,j+3)$, it can pass vertically through all the columns $C_j$ to $C_{j+3}$. In addition, if a robot enters horizontally the pattern $P_n$ in position $(1,j)$, it can leave it horizontally on rows $R_1$ and $R_{n}$ in either direction.
\end{lem}

A proof of this result is given by Figure \ref{fig:patternWP}.

\begin{center}
\begin{figure}[h]
\resizebox{4cm}{!}{
\input{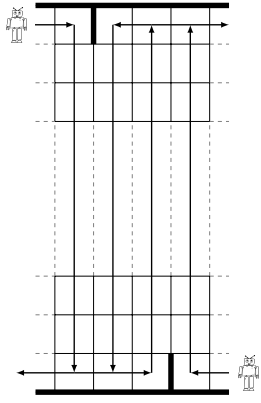}
}
\caption{Pattern of size $n\times 4$ for \IS{W}{P} on grid}
\label{fig:patternWP}
\end{figure}
\end{center}

\begin{theo}\label{theo:WP}
For any $n\geq m\geq 4$, we have $\opt{W}{P}{G_{n,m}}=\rho(P_{m-2})$.\\
For $m=3$ and $n\geq m$, we have $\opt{W}{P}{G_{n,m}}=1$.\\
For $m\leq 2$ and $n\geq m$, we have $\opt{W}{P}{G_{n,m}}=0$.\\
\end{theo}

\begin{proof}
Let $n\geq m\geq 4$. A solution is build by gluing $\lceil \frac{m-2}{4}\rceil$ copies of $P_n$ starting from column $C_2$ and ending in $C_{m-1}$ (leaving columns $C_1$ and $C_m$ empty). In this process, if $(m-2)\not\equiv 0\pmod 4$, then remove the $(m-2)\pmod 4$ last columns of the rightmost copy of $P_n$. Figure~\ref{fig:patternWP} ensures that all the columns (except the last ones) of all the copies are visited by entering from row $R_1$. It also shows that all the rightmost columns of each pattern can be visited by entering from the right on row $R_n$. Note that it is also true for the rightmost copy of $P_n$ since the robot can move vertically in column $C_m$.\\
In the case where $m=3$, it is straightforward to see that from position $(1,1)$, the robot cannot visit all the squares if there is no wall. A vertical wall adjacent to position $(1,2)$ is however sufficient.
\end{proof}

\begin{rem}
Note that Theorem~\ref{theo:WP} remains valid for any starting position. It is not hard to see that the pattern of Figure~\ref{fig:patternWP} ensures the robot to reach $(1,1)$ from any position.
\end{rem}

\subsection{\IS{W}{S}}\label{sec:isws}

In what follows, we consider the game \IS{}{} played with walls on a grid where the robot has to stop on each square. Constructions and proofs are similar to the ones of Subsection~\ref{sec:isbs}. We obtain the following lower bound for $\opt{W}{S}{G_{n,m}}$:

\begin{prop}
For each $n\geq m \geq 3$, we have $\opt{W}{S}{G_{n,m}}\geq \rho(G_{n-2,m-2}).$
\end{prop} 

\begin{proof}
Consider an optimal solution $W$ of \IS{W}{S} on $G_{n,m}$.
Let $(i,j)$ be an inner square of the grid, i.e $2\leq i \leq n-1$ and  $2\leq j\leq m-1$.
There must a wall adjacent to $(i,j)$ otherwise the robot cannot stop on it. 
Let $G'$ be the grid induced by the inner squares. It is isomorphic to $G_{n-2,m-2}$.
Let $W_1$ be the set of walls that are adjacent to two inner squares of $G_{n,m}$. 
Let $V_b$ be the set of squares of $G'$ that are not adjacent to a wall of $W_1$. Note that all these squares are on the border of $G'$ and for each square of $V_b$, there is a wall between this square and the adjacent square on the border of $G_{n,m}$. Hence one can add to $W_1$ one wall adjacent to each square of $V_b$ and this new set $W'$ will have size at most $|W|$. The set $W'$  naturally corresponds to a set of edges of $G'$ that cover the vertices of $G'$. Hence we have:
$$\opt{W}{S}{G_{n,m}}=|W|\geq |W'|\geq \rho(G_{n-2,m-2}).$$
\end{proof}

We now give a construction of a solution of \IS{W}{S} on $G_{n,m}$.

\begin{prop}\label{prop:cons-WS}
For any $n\geq m \geq 3$, we have  $$\opt{W}{S}{G_{n,m}}\leq \rho(G_{n-2,m-2})+\frac{3(m+n)}{4} +3.$$ 
\end{prop}

\begin{proof}
Let $n\geq m\geq 3$.
We construct a solution $W$ of \IS{W}{S} in three steps. The construction is illustrated on Figures \ref{fig:consWS97} and \ref{fig:rectWS}.

We denote by $(i,j)-(i,j+1)$ the vertical wall between the square $(i,j)$ and $(i,j+1)$ and use similar notations for horizontal walls. An {\em inner wall} is a wall that is not touching the border (if it is a vertical wall, it means that $2\leq j \leq m-2$ and $i$ is not constrained in the previous notation).

Let first $W_1$ be the set of inner walls that are the translation of the walls $(2,2)-(2,3)$ and $(3,2)-(4,2)$ by linear combinations of the three vectors $\{(1,1), (0,4), (4,0)\}$ (see the inner walls of Figure \ref{fig:consWS97}). Note that $W_1$ covers all the inner squares of $G_{n,m}$ and that each square is covered exactly once. Some walls of $W_1$ are located between an inner square and a border square. There are at most $2\left(\left\lceil \frac{n-2}{4} \right \rceil+\left\lceil \frac{m-2}{4} \right \rceil\right)\leq \frac{m+n}{2}+2$ such walls. Hence $|W_1|\leq \rho(G_{n-2,m-2})+\frac{m+n}{4}+1$.

The walls of $W_1$ will be enough to ensure that the robot will stop on almost each inner square as soon as it will enter to some of them. Indeed, if the robot enters in a rectangle of size $3\times 2$ that has four walls on its border (see as an example the gray rectangle on Figure \ref{fig:consWS97}), then it will be able to stop in the four corners of this rectangle and thus to go out of the rectangle from these four places (see Figure \ref{fig:rectWS}). To ensure that the robot will enter in the rectangles of the border of the grid and stop on the border, we add some walls that will form the set $W_2$. For each wall of $W_1$ located between an inner square and a border square, we add a wall in $W_2$ in the following way:

\begin{itemize}
\item for each $(1,j)-(2,j)\in W_1$, we add the wall $(1,j)-(1,j+1)$ to $W_2$;
\item for each $(n-1,j)-(n,j)\in W_1$, we add the wall $(n,j-1)-(n,j)$ to $W_2$;
\item for each $(i,1)-(i,2)\in W_1$, we add the wall $(i,1)-(i+1,1)$ to $W_2$;
\item for each $(i,m-1)-(i,m)\in W_1$, we add the wall $(i-1,m)-(i,m)$ to $W_2$.
\end{itemize}

We furthermore add the wall $(2,1)-(3,1)$ to $W_2$. Finally, if $(n-2,m-1)-(n-1,m-1)$ (respectively $(n-1,m-2)-(n-1,m-1)$) is a wall, we add the wall $(n,m-2)-(n,m-1)$ (resp. $(n-2,m)-(n-1,m)$) to $W_2$. 
Note that the first case holds if and only if $n-m\equiv 2\bmod 4$ and whereas the second case holds if and only if $n-m\equiv 3\bmod 4$.

The final set of walls $W_1\cup W_2$ is a solution of \IS{W}{S} of size at most $$\rho(G_{n-2,m-2})+\frac{3(m+n)}{4}+3.$$
\end{proof}

\begin{center}
\begin{figure}[h]
\resizebox{4cm}{!}{
\input{WS-cons97}
}
\caption{Solution of \IS{W}{S} on $G_{9,7}$}
\label{fig:consWS97}
\end{figure}
\end{center}

\begin{center}
\begin{figure}[h]
\resizebox{3cm}{!}{
\input{WS-rectangle}
}
\caption{Pattern of the solution of \IS{W}{S}}
\label{fig:rectWS}
\end{figure}
\end{center}

By combining the lower and uppers bounds and noticing that $\rho(G_{n,m})=\left\lceil \frac{nm}{2}\right \rceil$, we obtain an asymptotic value for $\opt{W}{S}{G_{n,m}}$:

\begin{cor}
If $m$ and $n$ are large enough, $\opt{W}{S}{G_{n,m}}=\frac{mn}{2}+O(m+n)$.
\end{cor}

As for Subsection~\ref{sec:isbs}, we did not try to find the exact value of $\opt{W}{S}{G_{n,m}}$.

\begin{rem}
The solution built in Proposition \ref{prop:cons-WS} remains valid for any starting position except the part of the border around the bottom-right corner when $n-m\equiv 0 \bmod 4$ or $n-m\equiv 1 \bmod 4$.
\end{rem}

\end{document}